\documentclass[11pt]{amsart}

\usepackage{amsmath,amssymb,amsthm}
\usepackage[a4paper,margin=2.8cm]{geometry}
\usepackage[colorlinks=true,linkcolor=blue,citecolor=blue,urlcolor=blue]{hyperref}

\usepackage[normalem]{ulem} %PARA TACHAR

\newtheorem{theorem}{Theorem}[section]
\newtheorem{proposition}[theorem]{Proposition}
\newtheorem{lemma}[theorem]{Lemma}
\newtheorem{corollary}[theorem]{Corollary}
\theoremstyle{definition}
\newtheorem{definition}[theorem]{Definition}
\theoremstyle{remark}
\newtheorem{remark}[theorem]{Remark}

\numberwithin{equation}{section}

\newcommand{\R}{\mathbb{R}}
\newcommand{\Sf}{\mathbb{S}}
\newcommand{\ip}[2]{\langle #1,#2\rangle}
\DeclareMathOperator{\tr}{tr}
\DeclareMathOperator{\Hess}{Hess}
\DeclareMathOperator{\dv}{div}
\DeclareMathOperator{\Dv}{Div}

\newcommand{\cambio}[1]{{\color{black}#1}} %para quitar los cambios en rojo
\def\a{\vec{a}}

\begin{document}

\title[Bernstein's theorem via Omori--Yau]{A short proof of Bernstein's theorem\\ via the Omori--Yau maximum principle} 

\author{Luis J. Al\'{\i}as}\thanks{Corresponding author: Luis J. Al\'{\i}as}
\address{Departamento de Matem\'aticas, Universidad de Murcia, E-30100 Espinardo, Murcia, Spain}
\email{ljalias@um.es}

\author{Miguel A. Mero\~no}
\address{Departamento de Matem\'aticas, Universidad de Murcia, E-30100 Espinardo, Murcia, Spain}
\email{mamb@um.es}

\subjclass[2020]{Primary 53A10; Secondary 53C42, 35B50}
\keywords{Bernstein theorem, minimal surfaces, Omori--Yau maximum principle, angle function}

\begin{abstract}
We give a short, self-contained proof of Bernstein's theorem: every entire minimal graph in
$\R^3$ is a plane. The only \cambio{analytic} global tool is the Omori--Yau maximum principle, which holds on
entire minimal graphs with their induced metric. We apply it exactly once to a single explicit function $F$ built from the surface's angle function and principal curvatures. A pointwise differential inequality for $F$, combined with the Omori--Yau principle, forces it to attain its minimum value everywhere, which yields the planarity of the surface and hence Bernstein's theorem. As a by-product of our approach, we show that the same method, with a different function and with some extra work, also proves in a simple and self-contained way the generalization of Bernstein's theorem given by Osserman that a complete minimal surface in $\R^3$ whose normals omit a neighbourhood of some direction must be a plane.
\end{abstract}

\maketitle

\section{Introduction}

The classical theorem of Bernstein \cite{Bernstein} states the following.

\begin{theorem}[Bernstein]\label{thm:Bernstein}
Let $u\in C^{2}(\R^{2})$ be a solution of the minimal surface equation
\begin{equation}\label{eq:mse}
\Dv\left(\frac{D u}{\sqrt{1+|D u|^{2}}}\right)=0\qquad\text{on }\R^{2},
\end{equation}
\textcolor{black}{where $D$ and $\Dv$ denote, respectively, the gradient and divergence operators on $\R^2$}. Then $u$ is an affine function.
\end{theorem}

Many proofs of Theorem~\ref{thm:Bernstein} are known; see for instance
\cite{Osserman-book, Heinz, Chern, dCP, FCS, Pogorelov} and the references therein. In this note
we show that it follows from a single application of the Omori--Yau maximum principle
\cite{Omori, Yau} to one explicit function. We prove the following result.

\begin{theorem}\label{thm:main}
Let \textcolor{black}{$\psi\colon\Sigma\to\R^{3}$} be a minimal immersion of a connected, oriented surface with
Gauss map $N$, and assume that the Omori--Yau maximum principle holds on $\Sigma$ with its
induced metric. 
If there is a unit vector $\a\in\R^{3}$ such that $\nu:=\ip{N}{\a}>0$ on $\Sigma$, then $\psi(\Sigma)$ is contained in a plane.
\end{theorem}

The Omori--Yau principle holds on every properly immersed minimal surface of $\R^{3}$, and in
particular on every entire minimal graph (Proposition~\ref{prop:OY-graphs}). \textcolor{black}{An entire graph of the form $z=u(x,y)$, with $u\in C^{2}(\R^{2})$, also has $\ip{N}{\vec{e}_3}>0$, since
\[
N(x,y)=\frac{1}{\sqrt{1+|D u|^{2}}}\left(-\tfrac{\partial u}{\partial x}, -\tfrac{\partial u}{\partial y}, 1 \right).
\]}
Theorem~\ref{thm:Bernstein} therefore follows at once from
Theorem~\ref{thm:main}.

The proof rests on the function
\begin{equation}\label{eq:F-intro}
F=\frac{\rho}{\sqrt{\rho^{2}+\kappa^{2}}}\,,\qquad \rho=\nu(1+\nu),\qquad
\kappa^{2}=-K=\tfrac12|A|^{2},
\end{equation}
where $K$ is the Gauss curvature and $|A|$ the norm of the second fundamental form. The
function $F$ is smooth and takes values in $(0,1]$, and $F=1$ exactly at the points where
$A$ vanishes. \cambio{See Remark \ref{rem:exponent} to find some intuition behind the choice of this function.} We show that $F$ satisfies the pointwise inequality
\begin{equation}\label{eq:key-intro}
(1-F^{2})^{2}\ \le\ 3|\nabla F|^{2}-F\,\Delta F,
\end{equation}
and the Omori--Yau principle then yields $F\equiv 1$. This type of argument goes back to Cheng and
Yau \cite{CY}. Apart from the Omori--Yau principle, every ingredient is a
pointwise identity of surface theory.

On the other hand, and as a generalization of Bernstein's theorem, Nirenberg conjectured that a complete minimal surface in $\R^{3}$ whose normals omit a
neighbourhood of some direction must be a plane. Osserman proved this conjecture in
\cite{Osserman59}.

\begin{theorem}[Osserman]\label{thm:Osserman}
Let $\psi\colon\Sigma\to\R^{3}$ be a complete minimal immersion of a connected, oriented
surface with Gauss map $N$. If $N(\Sigma)$ omits a nonempty open subset of $\Sf^{2}$, then
$\psi(\Sigma)$ is a plane.
\end{theorem}

Theorem~\ref{thm:Osserman} contains Bernstein's theorem, since the Gauss image of an entire
graph lies in an open hemisphere, and it was the starting point of the study of the image of the Gauss map of complete minimal surfaces: the Gauss image of a complete
nonflat minimal surface omits at most six points \cite{Xavier}, and in
fact at most four points \cite{Fujimoto}, which is sharp. See \cite{Osserman-book} for the
classical approach through the Weierstrass representation and complex analysis.

As a by-product of our approach in this paper, we show that the same method, with a different function and with some extra work, also proves Theorem~\ref{thm:Osserman} in a simple and self-contained way. The core of the argument is the following result.

\begin{theorem}\label{thm:main2}
Let $\psi\colon\Sigma\to\R^{3}$ be a minimal immersion of a connected, oriented surface with
Gauss map $N$, and assume that the Omori--Yau maximum principle holds on $\Sigma$ with its
induced metric. If there are a unit vector $\a\in\R^{3}$ and a constant $c\in(-1,1)$ such that
\[
\nu:=\ip{N}{\a}>c\qquad\text{on }\Sigma,
\]
then $\psi(\Sigma)$ is contained in a plane.
\end{theorem}

The hypothesis of Theorem~\ref{thm:main2} says that the Gauss image lies in the open spherical
cap $D_c=\{q\in\Sf^{2}:\ip{q}{\a}>c\}$. For $c=0$ this is an open hemisphere and Theorem~\ref{thm:main2} is Theorem~\ref{thm:main}. Every nonempty open subset of $\Sf^{2}$ contains a closed
spherical cap, so if $N(\Sigma)$ omits an open set, then $N(\Sigma)$ lies in some $D_c$ with
$c\in(-1,1)$.

The proof of Theorem~\ref{thm:main2} rests on the function
\begin{equation}\label{eq:F-intro2}
F_c=\frac{\rho_c}{\sqrt{\rho_c^{2}+(1-c^{2})\kappa^{2}}}\,,\qquad \rho_c=(\nu-c)(1+\nu),
\end{equation}
with $F_0=F$ when $c=0$. The function
$F_c$ is smooth, takes values in $(0,1]$, and $F_c=1$ exactly at the points where $A$ vanishes. As in the case where $c=0$, we
show that $F_c$ satisfies the pointwise inequality
\begin{equation}\label{eq:key-intro2}
(1+c)^{2}(1-F_c^{2})^{2}\ \le\ 3|\nabla F_c|^{2}-F_c\,\Delta F_c,
\end{equation}
and the Omori--Yau principle then yields $F_c\equiv1$.

To deduce Theorem~\ref{thm:Osserman} from Theorem~\ref{thm:main2} we need the Omori--Yau
principle on a complete minimal surface whose Gauss image lies in $D_c$. This is the content of
Proposition~\ref{prop:OY-complete}. Apart from the final covering argument in the proof of Theorem \ref{thm:Osserman}, it is the only place where completeness is used, and also
the only place where an auxiliary metric appears: the metric $(1+\nu)^{2}g$, which is flat and
complete, so that the classical theorem of Omori and Yau applies to it. For properly immersed
surfaces no auxiliary metric is needed at all (Corollary~\ref{cor:proper}).

Section~\ref{sec:prelim} fixes notation and recalls the Omori--Yau principle.
Section~\ref{sec:identities} collects the pointwise identities, Section~\ref{sec:F}
introduces $F$ and proves \eqref{eq:key-intro}, and Section~\ref{sec:proofs} proves
Theorems~\ref{thm:main} and~\ref{thm:Bernstein}. Finally, Section~\ref{sec:Osserman} proves Theorem~\ref{thm:main2} and Theorem~\ref{thm:Osserman}.

\section{Preliminaries}\label{sec:prelim}

\subsection{Notation}
Let $\psi\colon\Sigma\to\R^{3}$ be an immersion of a connected oriented surface. We give
$\Sigma$ the induced metric $g=\ip{\ }{\ }$, with Levi-Civita connection $\nabla$,
gradient $\nabla$, Laplacian $\Delta=\tr\Hess=\dv\nabla$, and we write $\overline{\nabla}$ for the flat
connection of $\R^{3}$. Let $N\colon\Sigma\to\Sf^{2}$ be the Gauss map and let \textcolor{black}{$AX=-\overline{\nabla}_XN$ be
the shape operator, defined for every tangent vector field $X\in\mathfrak{X}(\Sigma)$. The Gauss formula and the Codazzi equation read, respectively,}
\[
\overline{\nabla}_XY=\nabla_XY+\ip{AX}{Y}N,\qquad (\nabla_XA)Y=(\nabla_YA)X, \qquad \textcolor{black}{\text{for } X,Y\in\mathfrak{X}(\Sigma),}
\]
and the Gauss equation gives $K=\det A$. From now on the immersion is \emph{minimal}, that
is, $\tr A=0$. Its principal curvatures are then $\pm\kappa$ with $\kappa\ge0$, and
\begin{equation}\label{eq:kappa}
\kappa^{2}=-K=\tfrac12|A|^{2}.
\end{equation}
Note that $\kappa^{2}$ is a smooth function, whereas $\kappa$ itself is in general only
continuous. Only $\kappa^{2}$ will enter the computations.

We fix a unit vector \textcolor{black}{$\a\in\R^{3}$ and consider the \emph{angle function} $\nu=\ip{N}{\a}$ and
the tangent part $\a^{\top}=\a-\nu N\in\mathfrak{X}(\Sigma)$ of $\a$, so that}
\begin{equation}\label{eq:norma}
\textcolor{black}{|\a^{\top}|^{2}=1-\nu^{2}.}
\end{equation}
We shall use repeatedly the chain rule
\begin{equation}\label{eq:chain}
\nabla\phi(u)=\phi'(u)\nabla u,\qquad
\Delta\phi(u)=\phi'(u)\Delta u+\phi''(u)|\nabla u|^{2}.
\end{equation}

\subsection{The Omori--Yau maximum principle}

\begin{definition}\label{def:OY}
We say that the \emph{Omori--Yau maximum principle} holds on a Riemannian manifold $M$ if,
for every $u\in C^{2}(M)$ bounded from below, there is a sequence $\{p_k\}\subset M$ such
that
\[
u(p_k)<\inf_M u+\frac1k,\qquad |\nabla u(p_k)|<\frac1k,\qquad \Delta u(p_k)>-\frac1k .
\]
\end{definition}

Omori \cite{Omori} and Yau \cite{Yau} proved that the Omori--Yau maximum principle holds on
every complete Riemannian manifold whose Ricci curvature is bounded from below. For minimal
surfaces in $\R^{3}$ no such bound is available a priori. In the properly immersed case we use the following
well-known fact.

\begin{proposition}[\cite{PRS-memoirs, AMR}]\label{prop:OY-graphs}
The Omori--Yau maximum principle holds on every properly immersed minimal surface of $\R^3$ with its induced metric. In particular, it holds on every entire minimal graph in $\R^3$.
\end{proposition}

This is a special case of the theorem of Pigola, Rigoli and Setti asserting that the
Omori--Yau maximum principle holds on properly immersed submanifolds of Euclidean space with
bounded mean curvature \cite{PRS-memoirs, AMR}. \textcolor{black}{It is obtained from their function-theoretic
criterion in \cite[Theorem 1.9]{PRS-memoirs} applied to $\gamma=|\psi|^{2}$. The function $\gamma$ satisfies $|\nabla\gamma|=2|\psi^{\top}|\le2\sqrt{\gamma}$ and, since the immersion is proper, it  also satisfies $\gamma\to+\infty$ as $p\to\infty$. Besides, minimality gives $\Delta\gamma=4$}. An entire graph is a closed embedded surface, hence it is properly embedded.

For complete surfaces with Gauss image in a spherical cap we prove the Omori--Yau principle in
Proposition~\ref{prop:OY-complete}.

\section{Pointwise identities}\label{sec:identities}
Throughout this section $\psi\colon\Sigma\to\R^{3}$ is a minimal immersion. All statements
are local.

\begin{lemma}\label{lem:basic}
The following identities hold:
\begin{enumerate}
\item[(a)] $A^{2}=\kappa^{2}\,\mathrm{Id}$;
\item[(b)] $\nabla\nu=-A(\a^{\top})$ and $|\nabla\nu|^{2}=\kappa^{2}(1-\nu^{2})$;
\item[(c)] $\Delta\nu=-|A|^{2}\nu=-2\kappa^{2}\nu$.
\end{enumerate}
\end{lemma}

\begin{proof}
(a) $A$ is a self-adjoint endomorphism of a $2$-dimensional space with $\tr A=0$. By the
Cayley--Hamilton theorem, $A^{2}=-(\det A)\,\mathrm{Id}=-K\,\mathrm{Id}=\kappa^{2}\,\mathrm{Id}$.

\textcolor{black}{(b) For every tangent vector field $X\in\mathfrak{X}(\Sigma)$, 
\[
X(\nu)=\ip{\overline{\nabla}_XN}{\a}=-\ip{AX}{\a}=-\ip{AX}{\a^{\top}}=-\ip{X}{A(\a^{\top})};
\]
in other words, $\nabla\nu=-A(\a^{\top})$. Hence, by (a) and (\ref{eq:norma}),}
\[
\textcolor{black}{|\nabla\nu|^{2}=\ip{A^{2}\a^{\top}}{\a^{\top}}=\kappa^{2}|\a^{\top}|^{2}=\kappa^{2}(1-\nu^{2}).}
\]

\textcolor{black}{(c) From $\a=\a^{\top}+\nu N$ and the fact that $\overline{\nabla}_X\a=0$, we get $\nabla_X\a^\top=\nu A(X)$
for every $X\in\mathfrak{X}(\Sigma)$. Thus, using the Codazzi equation,
\begin{eqnarray*}
\nabla_X(\nabla\nu) & = & -\nabla_X(A(\a^\top))=-(\nabla_XA)(\a^\top)-A(\nabla_X\a^\top)\\
{} & = & -(\nabla_{\a^\top}A)(X)-\nu A^2(X).
\end{eqnarray*}
Let $\{E_1,E_2\}$ be a local orthonormal frame of tangent vector fields on $\Sigma$. Then
\begin{eqnarray*}
\Delta\nu & = & \sum_{i=1}^{2}\ip{\nabla_{E_i}(\nabla\nu)}{E_i}=
-\tr(\nabla_{\a^\top}A)-\nu\tr(A^2)\\
{} & = & -\a^\top(\tr(A))-\nu|A|^2=-\nu|A|^2=-2\kappa^2\nu.
\end{eqnarray*}
This finishes the proof.}
\end{proof}

\begin{lemma}\label{lem:logs}
Wherever $\nu>-1$ we have 
\[
\Delta\log(1+\nu)=-\kappa^{2}.
\]
If moreover $\nu>0$ and $\rho=\nu(1+\nu)$, then
\[
\Delta\log\nu=-\kappa^{2}\,\frac{1+\nu^{2}}{\nu^{2}},\qquad
\Delta\log\rho=-\kappa^{2}\,\frac{1+2\nu^{2}}{\nu^{2}} .
\]
\end{lemma}

\begin{proof}
By \eqref{eq:chain} and Lemma~\ref{lem:basic},
\begin{align*}
\Delta\log(1+\nu)&=\frac{\Delta\nu}{1+\nu}-\frac{|\nabla\nu|^{2}}{(1+\nu)^{2}}
=-\frac{2\kappa^{2}\nu}{1+\nu}-\frac{\kappa^{2}(1-\nu)}{1+\nu}=-\kappa^{2},\\
\Delta\log\nu&=\frac{\Delta\nu}{\nu}-\frac{|\nabla\nu|^{2}}{\nu^{2}}
=-2\kappa^{2}-\frac{\kappa^{2}(1-\nu^{2})}{\nu^{2}}=-\kappa^{2}\,\frac{1+\nu^{2}}{\nu^{2}} .
\end{align*}
The third identity is the sum of the first two, since $\log\rho=\log\nu+\log(1+\nu)$.
\end{proof}

\begin{lemma}\label{lem:logkappa}
On the open set $U=\{p\in\Sigma:\textcolor{black}{\kappa^2(p)>0}\}=\{p\in\Sigma: A_p\neq0\}$ we have
\[
\Delta\log\kappa^{2}=-4\kappa^{2}=-2|A|^2.
\]
\end{lemma}

\begin{proof}
Let $p\in U$. Since $dN_p=-A_p$ has eigenvalues $\pm\kappa(p)\ne0$, it is invertible, and for
$v,w\in T_p\Sigma$ Lemma~\ref{lem:basic}(a) gives
\[
\ip{dN_p v}{dN_p w}=\ip{A_p^{2}v}{w}=\kappa^{2}(p)\ip{v}{w}.
\]
Thus $N\colon(U,\kappa^{2}g)\to\Sf^{2}$ is a local isometry, and by the Theorema Egregium
the metric $\kappa^{2}g$ has Gauss curvature $1$. On a surface, the Gauss curvature of
$e^{2\varphi}g$ is $e^{-2\varphi}(K-\Delta\varphi)$. With $\varphi=\tfrac12\log\kappa^{2}$
and $K=-\kappa^{2}$ this gives
\[
1=\kappa^{-2}\Big(-\kappa^{2}-\tfrac12\Delta\log\kappa^{2}\Big),
\]
that is, $\Delta\log\kappa^{2}=-4\kappa^{2}$.
\end{proof}

\begin{remark}\label{simons}
\cambio{Lemma \ref{lem:logkappa} is essentially Simons' classical formula for minimal surfaces \cite{Simons}, combined with the fact that for minimal surfaces the refined Kato inequality $|\nabla A|^2\ge 2|\nabla|A||^2$ is an equality. Indeed,
on the open set $U=\{p\in\Sigma:\kappa^2(p)>0\}=\{p\in\Sigma: A_p\neq0\}$ we can write 
$\frac{1}{2}|A|^2=e^{\log\kappa^2}$
and compute
\begin{eqnarray*}
\frac{1}{2}\Delta|A|^2 & = & e^{\log\kappa^2}\Delta\log\kappa^2+e^{\log\kappa^2}|\nabla\log\kappa^2|^2=-|A|^4+e^{-\log\kappa^2}|\nabla e^{\log\kappa^2}|^2\\
{} & = & -|A|^4+\frac{1}{2|A|^2}|\nabla|A|^2|^2=-|A|^4+2|\nabla|A||^2,
\end{eqnarray*}
which is Simons' formula $\frac12\Delta|A|^2=|\nabla A|^2-|A|^4$ with $|\nabla A|^2=2|\nabla|A||^2$.
In particular, this yields $\frac12\Delta|A|^2\geq -|A|^4$ on $U$. Moreover, at points where $A_p=0$ the function $|A|^2$ attains its minimum and it also holds $\frac12\Delta|A|^2(p)\geq 0=-|A_p|^4$. Therefore,
\begin{equation}\label{eq:simons}
\frac12\Delta|A|^2\geq -|A|^4  \text{ on } \Sigma.
\end{equation}}
\end{remark}

\section{The function \texorpdfstring{$F$}{F} }\label{sec:F}
From now on $\nu>0$ on $\Sigma$, so that $\nu(p)\in(0,1]$ and $\rho(p)=\nu(p)(1+\nu(p))\in(0,2]$ at every $p\in\Sigma$. We set
\begin{equation}\label{eq:f}
f=\frac{\kappa^{2}}{\rho^{2}}=\frac{|A|^{2}}{2\nu^{2}(1+\nu)^{2}}\ \ge 0 .
\end{equation}
It is a smooth function on $\Sigma$, and $\{ p\in\Sigma : f(p)>0\}=U$.
We now introduce the function which is the heart of the proof:
\begin{equation}\label{eq:F}
F=\frac{1}{\sqrt{1+f}}=\frac{\rho}{\sqrt{\rho^{2}+\kappa^{2}}}.
\end{equation}
It is smooth, it satisfies $0<F\le1$ with $F(p)=1$ if and only if $A_p=0$, and
\begin{equation}\label{eq:1-F2}
1-F^{2}=\frac{f}{1+f}=F^2f.
\end{equation}
\cambio{See Remark \ref{rem:exponent} to find some intuition behind the choice of these functions.}

\begin{lemma}[Key inequality]\label{lem:key}
On $\Sigma$ we have
\begin{equation}\label{eq:identity}
\frac{1}{2}F^{4}\,\Delta f=3|\nabla F|^{2}-F\,\Delta F,
\end{equation}
which implies
\begin{equation}\label{eq:key}
(1-F^{2})^{2}\ \le\ 3|\nabla F|^{2}-F\,\Delta F .
\end{equation}
\end{lemma}

\begin{proof}
We have $F=\phi(f)$ with $\phi(t)=(1+t)^{-1/2}$. Then
$\phi'(t)=-\tfrac12(1+t)^{-3/2}$ and $\phi''(t)=\tfrac34(1+t)^{-5/2}$, that is,
\[
\phi'(f)=-\tfrac12F^{3},\qquad \phi''(f)=\tfrac34F^{5}.
\]
By \eqref{eq:chain}, $\nabla F=-\tfrac12F^{3}\nabla f$, so $|\nabla f|^{2}=4F^{-6}|\nabla F|^{2}$, and
\[
\Delta F=-\tfrac12F^{3}\Delta f+\tfrac34F^{5}|\nabla f|^{2}
=-\tfrac12F^{3}\Delta f+3F^{-1}|\nabla F|^{2}.
\]
Multiplying by $F$ gives \eqref{eq:identity}. 

On the other hand, observe that on $U$, $\log f=\log\kappa^{2}-2\log\rho$. By Lemmas~\ref{lem:logkappa} and~\ref{lem:logs},
\begin{equation}\label{eq:laplalog}
\Delta\log f=-4\kappa^{2}+2\kappa^{2}\,\frac{1+2\nu^{2}}{\nu^{2}}
=\frac{2\kappa^{2}}{\nu^{2}}=2(1+\nu)^{2}f .
\end{equation}
Hence, by \eqref{eq:chain} applied to $f=e^{\log f}$, on $U$ we get
\[
\Delta f=f\big(\Delta\log f+|\nabla\log f|^{2}\big)\ \ge\ f\,\Delta\log f
=2(1+\nu)^{2}f^{2}.
\]
Moreover, if $p\in\Sigma\setminus U$ we have $f(p)=0=\min_\Sigma f$, so that, $\Hess f_p\ge0$ and $\Delta f(p)\ge0=2(1+\nu(p))^{2}f^{2}(p)$. Consequently, 
\begin{equation}\label{eq:laplaf}
\Delta f\ \ge\ 2(1+\nu)^{2}f^{2}  \quad \text{ on } \Sigma.
\end{equation}
Multiplying (\ref{eq:laplaf}) by $\dfrac{1}{2}F^4$ and using (\ref{eq:1-F2}) gives
\[
\frac{1}{2}F^4\Delta f\geq F^4(1+\nu)^{2}f^{2}\geq F^4f^2=(1-F^2)^2.
\] 
Using this in (\ref{eq:identity}) gives $(1-F^{2})^{2}\le3|\nabla F|^{2}-F\Delta F$, which is \eqref{eq:key}.
\end{proof}

\section{Bernstein's theorem via the Omori--Yau maximum principle}\label{sec:proofs}
\begin{proof}[Proof of Theorem~\ref{thm:main}]
The function $F$ in \eqref{eq:F} is smooth on $\Sigma$ and bounded from below, since $F>0$. By
the Omori--Yau maximum principle there is a sequence $\{p_k\}\subset\Sigma$ with
\[
\inf_\Sigma F\le F(p_k)<\inf_\Sigma F+\frac1k,\qquad |\nabla F(p_k)|<\frac1k,\qquad
\Delta F(p_k)>-\frac1k .
\]
We evaluate \eqref{eq:key} at $p_k$. Since $0<F(p_k)\le1$, we have
$-F(p_k)\Delta F(p_k)<F(p_k)/k\le1/k$, and therefore
\[
0\ \le\ \big(1-F(p_k)^{2}\big)^{2}\ \le\ 3|\nabla F(p_k)|^{2}-F(p_k)\,\Delta F(p_k)
\ <\ \frac{3}{k^{2}}+\frac1k .
\]
Letting $k\to\infty$ we get $F(p_k)^{2}\to1$, hence $F(p_k)\to1$ because $F>0$. On the other
hand $F(p_k)\to\inf_\Sigma F$. Therefore
\[
\inf_\Sigma F=1 .
\]
Since $F\le1$, it follows that $F\equiv1$. By \eqref{eq:1-F2}, $\kappa\equiv0$, so $A\equiv0$.
Then $dN=-A=0$, and $N\equiv N_0$ is constant because $\Sigma$ is connected. Finally, for
every tangent vector $X$ we have \textcolor{black}{$X\ip{\psi}{N_0}=\ip{X}{N_0}=0$}, so $\ip{\psi}{N_0}$ is constant
and \textcolor{black}{$\psi(\Sigma)$} lies in a plane orthogonal to $N_0$.
\end{proof}

\begin{proof}[Proof of Theorem~\ref{thm:Bernstein}]
By elliptic regularity $u$ is smooth (indeed real analytic). Let
\textcolor{black}{$\Sigma=\{(x,y,u(x,y)):(x,y)\in\R^{2}\}\subset\R^3$}, with $W=\sqrt{1+|Du|^{2}}$ and upward unit normal
$N=(-Du,1)/W$. Equation \eqref{eq:mse} states that $\Sigma$ is minimal. Take $\a=\vec{e}_3$,
so that $\nu=\ip{N}{\vec{e}_3}=1/W>0$. By Proposition~\ref{prop:OY-graphs} the Omori--Yau maximum
principle holds on $\Sigma$. By Theorem~\ref{thm:main}, $\Sigma$ lies in a plane, and since
it is an entire graph, $u$ is affine.
\end{proof}

\begin{remark}[On the choice of $f$ and $F$]\label{rem:exponent}
To apply the Omori--Yau principle we need a function $0<F\le1$, equal to $1$
exactly where $A=0$, and satisfying an inequality of the form
\begin{equation}\label{eq:keybeta}
C_2(1-F^2)^2\le C_1|\nabla F|^{2}-F\Delta F,
\end{equation}
for certain positive constants $C_1,C_2>0$. For building such an $F$, it is enough to find a non-negative function $f$ satisfying $\Delta f\geq Cf^2$, for a positive constant $C>0$, with $f=0$ exactly where $A=0$. In that case, one may try $F_\beta$ given by 
\[
F_\beta=\frac{1}{(1+f)^\beta}, \quad \text{ with } \beta>0.
\]
$F_\beta$ satisfies $0<F_\beta\leq 1$ with $F_\beta(p)=1$ if and only if $A_p=0$, and (\ref{eq:1-F2}) is replaced by
\begin{equation}\label{eq:1-Fbeta}
1-F_\beta^{1/\beta}=\frac{f}{1+f}=F_\beta^{1/\beta} f. 
\end{equation}
A computation as in the first part of Lemma~\ref{lem:key}, with $F_\beta=\phi_\beta(f)$ and $\phi_\beta(t)=(1+t)^{-\beta}$, gives
\begin{equation}\label{eq:1-Fbeta2}
\frac{\beta+1}{\beta}|\nabla F_\beta|^{2}-F_\beta\Delta F_\beta=\beta\,F_\beta^{\,2+1/\beta}\,\Delta f,
\end{equation}
and hence, using now $\Delta f\geq Cf^2$ and (\ref{eq:1-Fbeta}), we get
\[
\frac{\beta+1}{\beta}|\nabla F_\beta|^{2}-F_\beta\Delta F_\beta\geq C\beta F^{2+1/\beta}_\beta f^2=C\beta F_\beta^{2-1/\beta}(1-F_\beta^{1/\beta})^2.
\]
For $\beta=\frac12$ the weight $F_\beta^{2-1/\beta}$ is identically $1$, which gives
\eqref{eq:keybeta} with $C_1=3$ and $C_2=C/2$. For $0<\beta<\frac12$ the argument also works, because the weight blows up
as $F_\beta\to0$ and this rules out $\inf F_\beta=0$. For $\beta>\frac12$, for instance
$F_1=1/(1+f)$, the weight tends to $0$ as $F_\beta\to0$, and the Omori--Yau sequence only
yields $\inf F_\beta\in\{0,1\}$. Thus $\beta=\frac12$ is the largest exponent for which the
argument closes directly.

On the other hand, a natural candidate for $f$ would be $f=\frac12|A|^2=\kappa^2$, but it does not work by itself, since by (\ref{eq:simons}) $\Delta f\geq -4f^2$ has the wrong sign. However, by (\ref{eq:chain}) and Lemma \ref{lem:basic}, every function of $\nu$ satisfies
$\Delta h(\nu)=\kappa^{2}(Lh)(\nu)$, with 
\[
Lh(t)=(1-t^{2})h''(t)-2th'(t).
\]
Correcting $\kappa^{2}$ by a function of $\nu$ of the form $f=\kappa^{2}/\eta^2(\nu)$, where $\eta(t)$ is a positive smooth function on $(0,1]$, we obtain 
\begin{eqnarray*}
\Delta f & = & \Delta(e^{\log f}) = f\Delta\log f+f|\nabla\log f|^2 \\
{} & \ge & f(\Delta\log\kappa^2-2\Delta\log\eta(\nu))\\
{} & = & -2\kappa^2f(2+(L\log\eta)(\nu))
\end{eqnarray*}
on $U=\{p\in\Sigma : \kappa^2(p)>0\}$. Doing so, the factor
$\kappa^{2}$ cancels, and the desired inequality $\Delta f\ge Cf^{2}$ on $U$ reduces
to the one-variable inequality 
\begin{equation}\label{eq:Llogpsi}
-L(\log\eta)(t)\ge2+\frac{C}{2\eta^{2}(t)}, \forall t\in (0,1].
\end{equation}
The choice $\eta_1(t)=t$ produces
\[
-L(\log\eta_1)(t)=1+\frac{1}{t^2},
\]
which fails at (\ref{eq:Llogpsi}) for every $C>0$, while the choice $\eta_2(t)=1+t$ also fails, since
\[
-L(\log\eta_2)(t)=1.
\]
But the product $\eta=\eta_1\cdot\eta_2$ produces
\[
-L(\log\eta)(t)=-L(\log\eta_1)(t)-L(\log\eta_2)(t)=
2+\frac{1}{t^2}\geq 2+\frac{1}{t^2(1+t)^2}, \forall t\in(0,1]
\]
and gives (\ref{eq:Llogpsi}) with $C=2$. Finally, as in the proof of Lemma \ref{lem:key}, at points where $f(p)=0$ the function $f$ attains its minimum and $\Delta f(p)\ge0=Cf^2(p)$ for every $C$.
\end{remark}

\section{Osserman's theorem via the Omori--Yau maximum principle}\label{sec:Osserman}
\subsection{The function \texorpdfstring{$F_c$}{F}} 
\begin{lemma}
\label{lem:logs2}
Let $c\in(-1,1)$. On the open set $\{p\in\Sigma:\nu(p)>c\}$ we have
\[
\Delta\log(\nu-c)=-\kappa^{2}\,\frac{1-2c\nu+\nu^{2}}{(\nu-c)^{2}} .
\]
\end{lemma}
For the proof of Lemma~\ref{lem:logs2}, simply observe that by \eqref{eq:chain} and Lemma~\ref{lem:basic},
\[
\Delta\log(\nu-c)=\frac{\Delta\nu}{\nu-c}-\frac{|\nabla\nu|^{2}}{(\nu-c)^{2}}
=\frac{-2\kappa^{2}\nu(\nu-c)-\kappa^{2}(1-\nu^{2})}{(\nu-c)^{2}}
=-\kappa^{2}\,\frac{1-2c\nu+\nu^{2}}{(\nu-c)^{2}} .
\]

From now on we fix $c\in(-1,1)$ and assume that $\nu>c$ on $\Sigma$, so that $\nu(p)\in(c,1]$ and
\[
\rho_c(p)=(\nu(p)-c)(1+\nu(p))\in\big(0,\,2(1-c)\big]
\]
at every $p\in\Sigma$. We set
\begin{equation}\label{eq:f2}
f_c=\frac{(1-c^{2})\,\kappa^{2}}{\rho_c^{2}}=\frac{(1-c^{2})\,|A|^{2}}{2(\nu-c)^{2}(1+\nu)^{2}}\ \ge 0 ,
\end{equation}
and introduce the smooth function
\begin{equation}\label{eq:F2}
F_c=\frac{1}{\sqrt{1+f_c}}=\frac{\rho_c}{\sqrt{\rho_c^{2}+(1-c^{2})\kappa^{2}}},
\end{equation}
which satisfies $0<F_c\le1$ with $F_c(p)=1$ if and only if $A_p=0$, and
\begin{equation}\label{eq:1-F22}
1-F_c^{2}=\frac{f_c}{1+f_c}=F_c^{2}f_c.
\end{equation}
For $c=0$, $\rho_0=\nu(1+\nu)$, $f_0=f$ and $F_0=F$.
\cambio{The intuition behind the choice of the functions $f_c$ and $F_c$ is similar to the one explained in Remark \ref{rem:exponent}, now taking into account $\nu>c$.} 

\begin{lemma}[Key inequality for $F_c$]\label{lem:key2}
On $\Sigma$ we have
\begin{equation}\label{eq:key2}
(1+c)^{2}(1-F_c^{2})^{2}\ \le\ 3|\nabla F_c|^{2}-F_c\,\Delta F_c.
\end{equation}
\end{lemma}
The proof of Lemma~\ref{lem:key2} follows the same ideas and computations as that of Lemma~\ref{lem:key}, replacing $f$ by $f_c$ and $F$ by $F_c$. Since $F_c=\phi(f_c)$ with the same $\phi(t)=(1+t)^{-1/2}$, identity (\ref{eq:identity}) remains true for $F_c$ and $f_c$, that is, 
\begin{equation}\label{eq:identity2}
\frac{1}{2}F_c^{4}\,\Delta f_c=3|\nabla F_c|^{2}-F_c\,\Delta F_c.
\end{equation}
Observe also that on $\{p\in\Sigma:f_c(p)>0\}=U$
\[
\log f_c=\log(1-c^{2})+\log\kappa^{2}-2\log(\nu-c)-2\log(1+\nu).
\]
By Lemmas~\ref{lem:logkappa},~\ref{lem:logs} and~\ref{lem:logs2},
\begin{eqnarray}
\Delta\log f_c & = & -4\kappa^{2}+2\kappa^{2}\,\frac{1-2c\nu+\nu^{2}}{(\nu-c)^{2}}+2\kappa^{2}
=\frac{2\kappa^{2}}{(\nu-c)^{2}}\Big((1-2c\nu+\nu^{2})-(\nu-c)^{2}\Big)\nonumber\\
{} & = & \frac{2(1-c^{2})\kappa^{2}}{(\nu-c)^{2}}=2(1+\nu)^{2}f_c .\label{eq:laplalog2}
\end{eqnarray}
Hence, by \eqref{eq:chain} applied to $f_c=e^{\log f_c}$, on $U$ we get
\[
\Delta f_c=f_c\big(\Delta\log f_c+|\nabla\log f_c|^{2}\big)\ \ge\ f_c\,\Delta\log f_c=2(1+\nu)^{2}f_c^{2}.
\]
Moreover, at points where $f_c(p)=0=\min_\Sigma f_c$,
$\Delta f_c(p)\ge0=2(1+\nu(p))^{2}f_c^{2}(p)$, so that
\begin{equation}\label{eq:laplaf2}
\Delta f_c\ \ge\ 2(1+\nu)^{2}f_c^{2}\qquad\text{on }\Sigma.
\end{equation}
Since $\nu>c>-1$, we have $(1+\nu)^{2}>(1+c)^{2}$. Multiplying \eqref{eq:laplaf2} by $\frac12F_c^{4}$
and using \eqref{eq:1-F22} gives
\[
\frac12F_c^{4}\Delta f_c\ \ge\ F_c^{4}(1+\nu)^{2}f_c^{2}\ \ge\ (1+c)^{2}F_c^{4}f_c^{2}=(1+c)^{2}(1-F_c^{2})^{2}.
\]
Using this in \eqref{eq:identity2} gives \eqref{eq:key2}.

\subsection{Proofs of the theorems}

\begin{proof}[Proof of Theorem~\ref{thm:main2}]
The function $F_c$ in \eqref{eq:F2} is smooth on $\Sigma$ and bounded from below, since $F_c>0$. By
the Omori--Yau maximum principle there is a sequence $\{p_k\}\subset\Sigma$ with
\[
\inf_\Sigma F_c\le F_c(p_k)<\inf_\Sigma F_c+\frac1k,\qquad |\nabla F_c(p_k)|<\frac1k,\qquad
\Delta F_c(p_k)>-\frac1k .
\]
We evaluate \eqref{eq:key2} at $p_k$. Since $0<F_c(p_k)\le1$, we have
$-F_c(p_k)\Delta F_c(p_k)<F_c(p_k)/k\le1/k$, and therefore
\[
0\ \le\ (1+c)^{2}\big(1-F_c(p_k)^{2}\big)^{2}\ \le\ 3|\nabla F_c(p_k)|^{2}-F_c(p_k)\,\Delta F_c(p_k)
\ <\ \frac{3}{k^{2}}+\frac1k .
\]
Since $1+c>0$, letting $k\to\infty$ we get $F_c(p_k)^{2}\to1$, hence $F_c(p_k)\to1$ because $F_c>0$.
On the other hand $F_c(p_k)\to\inf_\Sigma F_c$. Therefore
\[
\inf_\Sigma F_c=1 .
\]
Since $F_c\le1$, it follows that $F_c\equiv1$. By \eqref{eq:1-F22}, $\kappa\equiv0$, so $A\equiv0$. The proof then finishes as in the proof of Theorem~\ref{thm:main}.
\end{proof}

\begin{corollary}\label{cor:open}
Let $\psi\colon\Sigma\to\R^{3}$ be a minimal immersion of a connected, oriented surface, and
assume that the Omori--Yau maximum principle holds on $\Sigma$ with its induced metric. If
$N(\Sigma)$ omits a nonempty open subset of $\Sf^{2}$, then $\psi(\Sigma)$ is contained in a plane.
\end{corollary}

\begin{proof}
Let $W\subset\Sf^{2}$ be a nonempty open set with $N(\Sigma)\cap W=\emptyset$. Choose $q_0\in W$
and $\delta\in(0,2)$ such that the spherical cap $\{q\in\Sf^{2}:\ip{q}{q_0}>1-\delta\}$ is contained
in $W$. Then $\ip{N}{q_0}\le1-\delta$ on $\Sigma$. Take $\a=-q_0$ and $c=-1+\delta/2\in(-1,1)$. Then
\[
\nu=\ip{N}{\a}=-\ip{N}{q_0}\ \ge\ -1+\delta\ >\ c\qquad\text{on }\Sigma,
\]
and Theorem~\ref{thm:main2} applies.
\end{proof}

\begin{corollary}\label{cor:proper}
A properly immersed minimal surface of $\R^{3}$ whose Gauss image omits a nonempty open subset
of $\Sf^{2}$ is contained in a plane.
\end{corollary}

\begin{proof}
By Proposition~\ref{prop:OY-graphs} the Omori--Yau maximum principle holds on $\Sigma$, and
Corollary~\ref{cor:open} applies.
\end{proof}

Corollary~\ref{cor:proper} is, of course, contained in Theorem~\ref{thm:Osserman}, because a
properly immersed surface is complete. Its proof, however, uses only the induced metric. To
reach complete surfaces we need the following proposition.

\begin{proposition}\label{prop:OY-complete}
Let $\psi\colon\Sigma\to\R^{3}$ be a complete minimal immersion and assume that there are a unit
vector $\a\in\R^{3}$ and $\delta\in(0,2]$ such that $1+\nu\ge \delta$ on $\Sigma$, where
$\nu=\ip{N}{\a}$. Then the Omori--Yau maximum principle holds on $\Sigma$ with its induced
metric.
\end{proposition}

\begin{proof}
Let $\mu=1+\nu$, so that $0<\delta\le\mu\le2$. Consider on $\Sigma$ the metric
$\widehat g=\mu^{2}g$, with gradient $\widehat\nabla$ and Laplacian $\widehat\Delta$. Let us see that $(\Sigma, \widehat g)$ is a complete flat surface. 
As recalled in the proof of Lemma~\ref{lem:logkappa}, the Gauss
curvature of $\widehat g=e^{2\varphi}g$ is $\widehat K=e^{-2\varphi}(K-\Delta\varphi)$. With $\varphi=\log\mu=\log(1+\nu)$, Lemma~\ref{lem:logs} gives
\[
\widehat K=(1+\nu)^{-2}\big(K-\Delta\log(1+\nu)\big)=(1+\nu)^{-2}\big(-\kappa^{2}+\kappa^{2}\big)=0 .
\]
On the other hand, since $\widehat g\ge\delta^{2}g$, the $\widehat g$-length of any
curve is at least $\delta$ times its $g$-length. Hence every closed $\widehat g$-bounded subset of
$\Sigma$ is closed and $g$-bounded, and therefore compact because $g$ is complete. By the
Hopf--Rinow theorem, $\widehat g$ is complete. Therefore, since $\widehat g$ is complete with zero Gauss
curvature, the theorem of Omori and Yau \cite{Omori, Yau} shows that the Omori--Yau maximum
principle holds on $(\Sigma,\widehat g)$.

On a surface, a conformal change $\widehat g=\mu^{2}g$ gives
$\widehat\nabla u=\mu^{-2}\nabla u$ and $\widehat\Delta u=\mu^{-2}\Delta u$; in particular
$|\widehat\nabla u|_{\widehat g}=\mu^{-1}|\nabla u|$. Let $u\in C^{2}(\Sigma)$ be bounded from below
and let $\{q_j\}$ be a sequence given by the Omori--Yau principle on $(\Sigma,\widehat g)$. Since
$\mu\le2$,
\[
u(q_j)<\inf_\Sigma u+\frac1j,\qquad |\nabla u(q_j)|=\mu|\widehat\nabla u(q_j)|_{\widehat g}<\frac2j,\qquad
\Delta u(q_j)=\mu^{2}\,\widehat\Delta u(q_j)>-\frac4j .
\]
The sequence $p_k=q_{4k}$ satisfies the conditions of Definition~\ref{def:OY} for $g$, which means that the Omori--Yau maximum principle also holds on $(\Sigma,g)$.
\end{proof}

\begin{remark}
In Proposition~\ref{prop:OY-complete} the auxiliary metric $\widehat g$ is used only to establish the
Omori--Yau principle for the induced metric, in the same way as Proposition~\ref{prop:OY-graphs}
is used for proper surfaces. The proof of Theorem~\ref{thm:main2}, and the function $F_c$, live
entirely in the induced metric.
\end{remark}

\begin{proof}[Proof of Theorem~\ref{thm:Osserman}]
As in the proof of Corollary~\ref{cor:open}, choose $q_0$, $\delta\in(0,2)$, $\a=-q_0$ and
$c=-1+\delta/2$, so that $\nu=\ip{N}{\a}\ge-1+\delta>c$, that is, $1+\nu\ge\delta$ on $\Sigma$. By
Proposition~\ref{prop:OY-complete}, the Omori--Yau maximum principle holds on $\Sigma$ with its
induced metric. By Theorem~\ref{thm:main2}, $\psi(\Sigma)$ is contained in a plane $P$. Then
$\psi\colon\Sigma\to P$ is an isometric immersion from a complete surface into a connected one; hence by \cite[Chap.~IV, Theorem~4.6]{KN} $\psi(\Sigma)=P$ and $\psi\colon\Sigma\to P$ is a covering map. Since $P$ is simply connected, $\psi$ is a
diffeomorphism onto $P$.
\end{proof}

\begin{remark}[Sharpness]\label{rem:sharp}
The hypothesis cannot be weakened to omitting a single point. Enneper's surface is a complete
nonflat minimal surface whose Gauss map omits exactly one point of $\Sf^{2}$; in the notation of
Theorem~\ref{thm:main2} it corresponds to the excluded value $c=-1$, where the constant in
\eqref{eq:key2} vanishes. More generally, the Gauss map of a complete nonflat minimal surface can
omit up to four points, and four is sharp \cite{Fujimoto}; Scherk's doubly periodic surface
omits four points. The method of this paper therefore reaches exactly the class of surfaces whose
Gauss image omits an open set.
\end{remark}

\section*{Declarations}
\subsection*{Funding} L.J. Al\'{\i}as is partially supported by the grant PID2025-167696NB-I00, funded by MICIU/AEI/10.13039/501100011033, Spain, and by the grant 21899/PI/22, funded by Fundaci\'{o}n S\'{e}neca, Regional Agency for Science and Technology, Comunidad Aut\'{o}noma de la Regi\'{o}n de Murcia, Spain, within the framework of the Regional Programme in Promotion of the Scientific and Technical Research (Action Plan 2022). M.A. Mero\~no is partially supported by the grant PID2023-151075OA-I00, funded by MCIU/AEI/10.13039/501100011033/ FEDER, UE, Spain.

\subsection*{Conflict of Interest/Competing interests} The authors have no conflict of interest or competing interests to declare that are relevant to the content of this article.

\end{document}